\documentclass[12pt]{amsart}

\usepackage{amsmath, amscd, graphicx, latexsym, hyperref, times, rlepsf}

\oddsidemargin 0in \theoremstyle{plain} \topmargin 0in

\oddsidemargin .25in \theoremstyle{plain}

\newtheorem{Thm}{Theorem}

\newtheorem{Prop}[Thm]{Proposition}

\newtheorem{Rem}[Thm]{Remark}

\newcommand{\os}{the contact Ozsv\'{a}th-Szab\'{o} invariant }

\newcommand{\bfz}{{\mathbb{Z}}}

\newcommand{\OB}{{\textrm{OB}}}

\def\p{\partial}

\def\v{\vskip.12in}
\def\a{\alpha}
\def\b{\beta}

\begin{document}

\title[]{On the contact Ozsv\'{a}th-Szab\'{o} invariant}

\author{Tolga Etg\"u}

\author{Burak Ozbagci}

\begin{abstract}

Sarkar and Wang proved that the hat version of Heegaard Floer
homology group of a closed oriented $3$-manifold is combinatorial
starting from an arbitrary nice Heegaard diagram and in fact every
closed oriented $3$-manifold admits such a Heegaard diagram.
Plamenevskaya showed that \os is combinatorial once we are given
an open book decomposition compatible with a contact structure.
The idea is to combine the algorithm of Sarkar and Wang with the
recent description of \os due to Honda, Kazez and Mati\'{c}. Here
we simply observe that the hat version of the Heegaard Floer
homology group and \os in this group can be combinatorially
calculated starting from a contact surgery diagram. We give
detailed examples pointing out to some shortcuts in the
computations.
\end{abstract}

\address{Department of Mathematics \\ Ko\c{c} University \\ Istanbul, Turkey}
\email{tetgu@ku.edu.tr} \email{bozbagci@ku.edu.tr}


\thanks{}

\v \v \v

\maketitle

\setcounter{section}{-1}


\section{Introduction}

We know that every closed contact $3$-manifold $(Y,\xi)$ can be
obtained by a contact $\pm 1$ surgery on a Legendrian link in the
standard contact $S^3$ (\cite{dg}).  It is often convenient to
describe $(Y, \xi)$ by a surgery diagram on the plane, i.e., by
the projection of a Legendrian link in the standard contact
$(\mathbb{R}^3, ker(dz+xdy)) $ onto the $yz$-plane with a $\pm
{1}$ surgery coefficient assigned to each component of the link.
Let $s_\xi$ denote the $Spin^c$ structure induced by $\xi$. In
order to calculate the Heegaard Floer homology group $\widehat{HF}
(-Y,s_\xi)$ and in particular to identify \os $c(\xi) \in
\widehat{HF} (-Y,s_\xi)$ we first find a suitable open book
decomposition compatible with $(Y, \xi)$ using the algorithm in
\cite{ao} (see also
\cite{g},\cite{s},\cite{p1},\cite{st},\cite{e},\cite{eo},\cite{a})
and then construct a compatible Heegaard diagram for $-Y$ as in
\cite{hkm} which also includes a description of a certain cycle
descending to $c(\xi)$ in $\widehat{HF} (-Y,s_\xi)$. Next we
convert this Heegaard diagram into a nice Heegaard diagram
(\cite{sw}) applying some finger moves without affecting the
homology class $c(\xi)$ --- no handle slides are necessary
\cite{p}. Finally we calculate $\widehat{HF}
(-Y,s_\xi)$ and $c(\xi) \in \widehat{HF} (-Y,s_\xi)$ by simply
counting certain squares and bigons in this nice Heegaard diagram.
In fact this procedure will allow us to calculate $\widehat{HF}
(-Y) \cong \widehat{HF} (Y) $, not just $\widehat{HF} (-Y,s_\xi)$.

We note that each step of the suggested combination of the above
algorithms can be quite involved and one would like to reduce the
calculations as much as possible by making certain choices. Here
we demonstrate the significance of a particular choice in
simplifying the calculations.

We assume that the reader is familiar with the basics of the Heegaard
Floer theory (see \cite{os1}, \cite{os}). We will work with $\mathbb{Z}_2$ coefficients in our
calculations throughout this paper.


\section{The contact Ozsv\'{a}th-Szab\'{o} invariant is combinatorial}

{\Thm Let $(Y,\xi)$ be a closed contact $3$-manifold described by
a contact surgery diagram on the plane.  We observe that the
$Spin^c$ structure $s_\xi$, Heegaard Floer homology groups
$\widehat{HF} (-Y,s_\xi) \subseteq \widehat{HF}(Y)$ and \os
$c(\xi) \in \widehat{HF} (-Y,s_\xi)$ can be
 calculated combinatorially.}

\begin{proof} Let $(Y,\xi)$ be a closed contact $3$-manifold described by
a contact surgery diagram on the plane, i.e., by the projection of
a Legendrian link in the standard contact $$(\mathbb{R}^3,
ker(dz+xdy)) \subset (S^3, \xi_{st})$$ onto the $yz$-plane with a
$\pm {1}$ surgery coefficient assigned to each component of the
link. First we use the algorithm in \cite{ao} (see also
\cite{g},\cite{s},\cite{p1},\cite{st},\cite{e},\cite{eo},\cite{a})
to find an explicit open book decomposition compatible with $(Y,
\xi)$. The idea in \cite{ao} is to embed the Legendrian surgery
link into the pages of an open book decomposition in $S^3$
compatible with its standard contact structure and then perform
the required contact surgeries to obtain an open book
decomposition $\OB_\xi$ of $Y$ compatible with the resulting
contact structure $\xi$.

Next we briefly recall (\cite{hkm}) how to get a Heegaard diagram
for $-Y$ which also includes a cycle that descends to \os $c(\xi)$
starting from a given open book decomposition  $\OB_\xi$ of $Y$
compatible with $\xi$. The open book decomposition $\OB_\xi$ can
be described as follows: Let $S$ denote the page and let $h : S
\to S$ denote the monodromy of $\OB_\xi$. Then $Y$ is homeomorphic
to $S \times [0,1] / \sim$, where the equivalence relation is
given by
$$ (p,1) \sim (h(p), 0), \;\;\; p\in S$$ $$ (p,t)\sim (p,t'),
\;\;\; p\in \p S; \;t,t' \in [0,1].$$

It is not too hard to see that $Y= H_1 \cup H_2$ is a Heegaard
splitting of $Y$, where $H_1= S\times [0,1/2] / \sim $ and $H_2=
S\times [1/2,1] / \sim$. Let $S_0$ and $S_{1/2}$ denote $S \times
\{0\}$ and $S \times \{1/2\}$ in  $H_1$, respectively. A basis
on a compact surface $S$ with boundary is just a collection
of properly embedded disjoint arcs $\{a_1, \ldots, a_n \}$ on $S$
such that when we cut $S$ along these arcs we get a single
polygon. Now we choose a basis $\{a_1, \ldots, a_n \}$ on the
page $S$ and choose a point $z$ in the polygonal region mentioned
above. Consider the closed surface $\Sigma = S_{1/2} \cup -S_0$
and glue the arc $a_i$ on $S_{1/2}$ with the arc $a_i$ on $-S_0$
to obtain a closed curve $\a_i$ on $\Sigma$. Let $b_i$ be an arc
which is isotopic to $a_i$ by a small isotopy so that the
following hold:

(1) The endpoints of $a_i$ are isotoped along $\p S_{1/2}$, in the
direction given by the orientation of $S_{1/2}$.

(2) The arcs $a_i$ and $b_i$ intersect transversely in one point
$x_i$ in the interior of $S_{1/2}$.

(3) If we orient $a_i$, and $b_i$ is given the induced orientation
from the isotopy, then the sign of the intersection of $a_i$ and
$b_i$ at $x_i$ is $+1$.

Then consider the arc $h(b_i)$ on $-S_0$ and glue the arc $b_i$
and $h(b_i)$ to get a closed curve $\b_i$ on $\Sigma$. If we let
$\a= \{\a_1, \ldots, \a_n\}$ and $\b= \{ \b_1, \ldots, \b_n \}$,
then $(\Sigma, \b,\a, z) $ is a Heegaard diagram for $-Y$, while
$(\Sigma, \a,\b, z) $ is a Heegaard diagram for $Y$. Moreover $X=
\{ x_1, \ldots, x_n \} \in \mathbb{T}_\a \cap \mathbb{T}_\b
\subset Sym^n(\Sigma)$ is a cycle in $ \widehat{CF} ( \Sigma, \b,
\a, z )$ which descends to \os $c(\xi) \in \widehat{HF} (-Y)$,
where $\mathbb{T}_\a = \a_1 \times \cdots \times \a_n$ and
$\mathbb{T}_\b = \b_1 \times \cdots \times \b_n$. Furthermore
there is a map from the set of generators $\mathbb{T}_\a \cap
\mathbb{T}_\b$ of $ \widehat{CF} ( \Sigma, \b, \a, z )$ to the set
of $Spin^c$ structures on $Y$. It turns out (\cite{os}) that the
special cycle $X$ corresponds to the $Spin^c$ structure $s_\xi$
induced by $\xi$. Therefore $c(\xi)$ belongs to $ \widehat{HF}
(-Y,s_\xi) \subseteq \widehat{HF} (-Y)$. We note that
$c_1(\xi)=c_1(s_\xi) \in H^2(Y;\mathbb{Z})$ can be calculated
(\cite{dgs}) combinatorially from a given contact surgery diagram
of $\xi$ (see page 195 in \cite{ozst}).

A connected component of the complement of $\alpha$ and $\beta$
curves in $\Sigma$ is called a region. Now we use the algorithm of
\cite{sw} to convert this Heegaard diagram into a nice  Heegaard
diagram, which we still denote by $( \Sigma, \b, \a, z )$,
so that all the regions on $\Sigma$ not including the base
point $z$ are bigons and squares. In general we would need to
apply finger moves and handle slides of the $\b$ curves in the
Heegaard diagram. Handle slides, fortunately, do not arise in our
case \cite{p} and a finger move corresponds to a certain kind of
isotopy of the $\b$ curves.

Recall that a domain is a formal linear combination of the regions
on $\Sigma$. A domain $D$ is called an empty embedded $2m$-gon, if

(1) $D$ has coefficients $0$ and $1$ everywhere,

(2) $D$ is topologically an embedded disk on $\Sigma$, with $2m$
vertices on its boundary,

(3) There is exactly one region with  coefficient $1$ around each
vertex on the $\p D$,

(4) $D$ does not contain any intersection points of $\a$ and $\b$
curves in its \emph{interior}.

Once we have a nice Heegaard diagram $(\Sigma, \b,\a, z) $, by
\cite{sw}, it is combinatorial to calculate the boundary map of
the Heegaard Floer chain complex. We just make a list of all the
generators and count all the empty embedded bigons and the empty
embedded squares on the Heegaard surface connecting these
generators by examining the diagram. Finally by using simple
linear algebra with $\mathbb{Z}_2$ coefficients we can compute
$\widehat{HF} (-Y)$. Here we emphasize that we can combinatorially
determine all the generators which are mapped to the distinguished
$Spin^c$ structure $s_\xi$, calculate $\widehat{HF} (-Y, s_\xi)$
and identify $c(\xi) \in \widehat{HF} (-Y, s_\xi)$.
\end{proof}


\section{The unique tight contact structure on $S^1 \times S^2$}

Consider the contact $3$-manifold $(Y, \xi)$ described by the
surgery diagram depicted in Figure~\ref{s1s2}. When we convert
this diagram into a smooth diagram (cf. Figure~\ref{sms1s2}) and
blow down the $-1$-curve, we immediately see that the underlying
$3$-manifold $Y$ is nothing but $S^1 \times S^2$. It is well-known
that there exists a
unique tight contact structure on $S^1 \times S^2$ up to isotopy \cite{elias}.

\begin{Prop}
The contact structure $\xi$ is the unique tight contact structure on
$S^1 \times S^2$.
\end{Prop}

\begin{proof}
Below we show that \os $c(\xi)$ is nontrivial. Therefore by a fundamental result in
\cite{os} $\xi$ is tight.
\end{proof}

{\Rem In particular, the unique tight contact structure on $S^1 \times S^2$
has nontrivial contact Ozsv\'{a}th-Szab\'{o} invariant. This was first proved in \cite{ls}.}

\begin{figure}[ht]
  \relabelbox \small {
  \centerline{\epsfbox{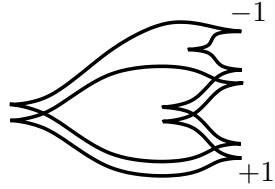}}}
  \relabel{1}{{$-1$}}
  \relabel{2}{{$+1$}}

  \endrelabelbox
        \caption{A contact surgery diagram}
        \label{s1s2}
\end{figure}

\begin{figure}[ht]
  \relabelbox \small {
  \centerline{\epsfbox{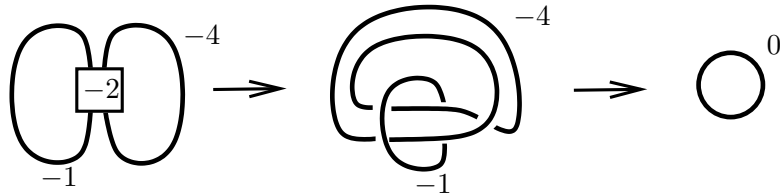}}}
  \relabel{1}{{$-1$}}
  \relabel{2}{{$-4$}}
 \relabel{3}{{$0$}}
 \relabel{4}{{$-2$}}
 \relabel{5}{{$-1$}}
 \relabel{6}{{$-4$}}

  \endrelabelbox
        \caption{Underlying $3$-manifold is $S^1 \times
S^2$}
        \label{sms1s2}
\end{figure}

First we would like to understand the homotopy class of $\xi$
considered as an oriented plane field and determine the $Spin^c$
structure $s_\xi$ induced by $\xi$.  We calculate the first Chern
class of $\xi$ as follows: Let $K_1$ and $K_2$ denote the $\pm
{1}$ surgery curves in Figure~\ref{s1s2}, respectively. Orient
these Legendrian knots and let $\mu_1$ and $\mu_2$ denote the
oriented meridians of $K_1$ and $K_2$, respectively. Then by
\cite{dgs} we have
$$PD(c_1(\xi))= rot(K_1)[\mu_1]+ rot(K_2)[\mu_2] = [\mu_1]+2[\mu_2] = 0 \in H_1(S^1
\times S^2, \mathbb{Z}) $$ where $PD$ denotes the Poincare dual
and $rot(K_i)$ denotes the rotation number of $K_i$. Moreover
since $H_1(S^1 \times S^2; \mathbb{Z})=\mathbb{Z}$ has no
$2$-torsion the $Spin^c$ structure $s_\xi$ is determined by
$c_1(\xi)$. In other words $s_\xi$ is the unique $Spin^c$
structure on $S^1 \times S^2$ whose first Chern class is trivial.

Our goal, however, is  to calculate $\widehat{HF} (-Y, s_\xi)$,
$\widehat{HF} (-Y)$ and in particular \os $c(\xi) \in \widehat{HF}
(-Y,s_\xi)$. By applying the techniques in \cite{e} we can find an
open book decomposition $\OB_\xi$ (see Figure~\ref{opbook})
compatible with $\xi$: First we start with the open book
decomposition $\OB_H$ induced by the positive Hopf link $H$ in
$S^3$, whose page is an annulus. Then we stabilize this open book
decomposition once and embed the +1 surgery curve onto a page.
Next we stabilize one more time and embed the -1 surgery curve
onto a page. Applying the required surgeries we get the desired
open book decomposition. Note that we get exactly the same open
book considered by Plamenevskaya in \cite{p}. By the lantern
relation on the four punctured sphere we know that the monodromy
of $\OB_\xi$ is a product of two right-handed Dehn twists and
hence $\xi$ is Stein fillable \cite{g}. Therefore we know that \os
of $\xi$ is nontrivial \cite{os}. In the following we will verify
this fact by the algorithm described in \cite{p}, but we will
choose three different bases to illustrate that this choice
is in fact crucial in calculations.


\subsection{Basis I}

\begin{figure}[ht]
  \relabelbox \small {
  \centerline{\epsfbox{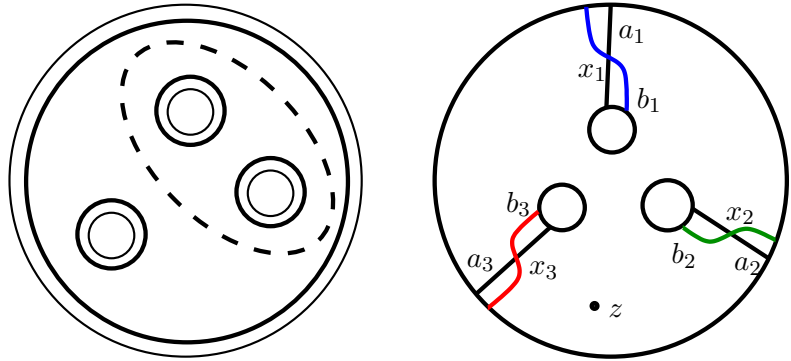}}}
  \relabel{1}{{$a_1$}}
  \relabel{2}{{$a_3$}}
  \relabel{3}{{$a_2$}}
\relabel{4}{{$b_1$}}
  \relabel{5}{{$b_3$}}
  \relabel{6}{{$b_2$}}
\relabel{a}{{$x_1$}}
  \relabel{c}{{$x_3$}}
  \relabel{b}{{$x_2$}}
  \relabel{d}{{$z$}}
  \endrelabelbox
        \caption{Left: Dehn twists about the solid curves are right-handed, while the Dehn twist
        about the dashed curve is left-handed. Right: A basis $\{a_1,a_2,a_3\}$
        on the page $S_{1/2}$, the arcs $\{b_1,b_2, b_3\}$, the intersection points
        $\{x_1, x_2, x_3 \}$, and the base point $z$. }
        \label{cutolga}

\end{figure}

First we take the basis $\{a_1, a_2, a_3\}$ on page $S$ which
is shown on the right in Figure~\ref{cutolga}. This is the basis
that was used in \cite{p}.  We observe that there are two
``bad" regions, one non-disk the other a hexagon. We divide each of
these regions into two square regions by a simple finger move
(\cite{p}) introducing a bigon in the process. The resulting curves are
depicted in Figure~\ref{nonsquare}. After this modification of the
Heegaard diagram there are nine regions which do not contain $z$.
These regions are denoted by $R_1,\ldots ,R_9$ and labelled by their indices in
Figure~\ref{nonsquare}.

\begin{figure}[ht]
  \relabelbox \small {
  \centerline{\epsfbox{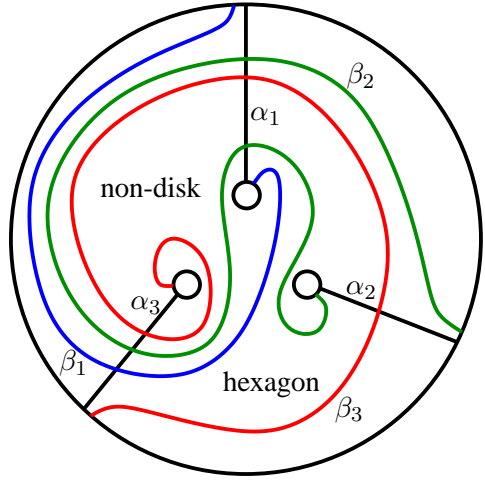}}}
  \relabel{2}{{non-disk}}
  \relabel{1}{{hexagon}}
\relabel{3}{{$\b_3$}}
  \relabel{4}{{$\b_2$}}
  \relabel{5}{{$\b_1$}}
\relabel{a}{{$\a_1$}}
 \relabel{b}{{$\a_2$}}
 \relabel{c}{{$\a_3$}}
  \endrelabelbox
        \caption{Bad regions are indicated on page $S_0$}
        \label{monolga}

\end{figure}

\begin{figure}[ht]
  \relabelbox \small {
  \centerline{\epsfbox{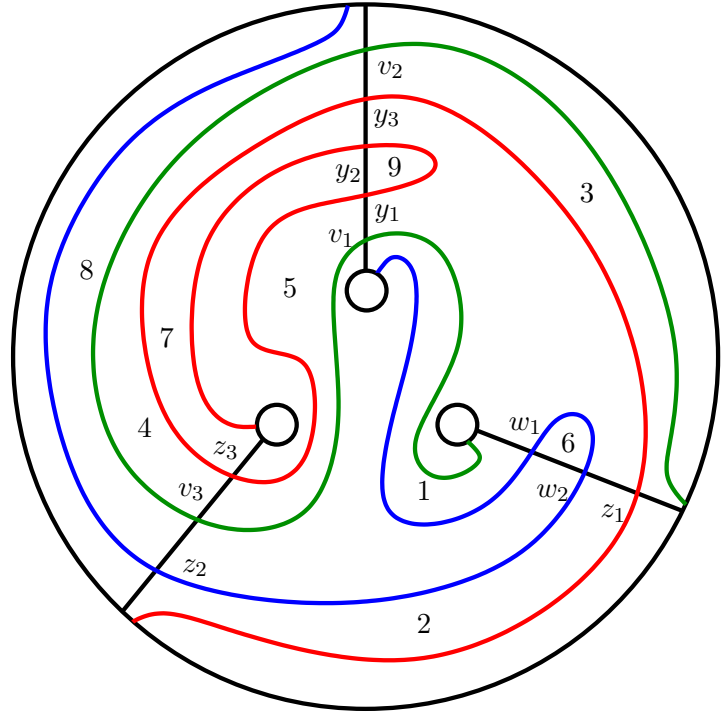}}}
  \relabel{1}{{$1$}}
  \relabel{2}{{$2$}}
  \relabel{3}{{$3$}}
 \relabel{4}{{$4$}}
  \relabel{5}{{$5$}}
  \relabel{6}{{$6$}}
\relabel{7}{{$7$}}
  \relabel{8}{{$8$}}
 \relabel{9}{{$9$}}

  \relabel{a}{{$v_1$}}
 \relabel{b}{{$y_1$}}
 \relabel{c}{{$y_2$}}
  \relabel{d}{{$y_3$}}
  \relabel{e}{{$v_2$}}
\relabel{f}{{$w_1$}}
 \relabel{g}{{$w_2$}}
\relabel{h}{{$z_1$}}
 \relabel{i}{{$z_3$}}
\relabel{j}{{$v_3$}}
 \relabel{k}{{$z_2$}}

  \endrelabelbox
        \caption{Finger moves}
        \label{nonsquare}

\end{figure}

Now by examining the intersections of $\alpha$ and $\beta$ curves
on $\Sigma$ we see that the generators of the Heegaard Floer chain
complex $\widehat{CF}(\Sigma, \b, \a, z) $ are $X=(x_1,x_2,x_3),
A=(x_1,z_1,v_3), B=(x_1,x_2,z_3), C_k=(y_k,x_2,z_2),
D_{ij}=(v_i,w_j,z_3), E_{ij}=(v_i,w_j,x_3),F_i=(v_i,z_1,z_2),
G_{kj}=(y_k,w_j,v_3), $ where $ 1 \leq i,j \leq 2 $ and $1 \leq k
\leq 3$. We calculated all the boundary maps induced by the empty
embedded bigons and empty embedded squares:

$\p X=0$

$\p A=B$ by $R_3+R_4$

$\p B=0$

$\p C_{1}=0$

$\p C_{2}=X+C_1$ by $R_4+R_7+R_8$ and  $R_9$

$\p C_{3}= B$ by $R_4+R_8$

$\p D_{11}=B+G_{11}+D_{12}$ by $R_1$, $R_5$ and $R_6$

$\p D_{12}=G_{12}$ by $R_5$

$\p D_{21}=D_{22}$ by $R_6$

$\p D_{22}=0$

$\p E_{11}=X+E_{12}$ by $R_1$ and $R_6$

$\p E_{12}=0$

$\p E_{21}=E_{22}$ by $R_6$

$\p E_{22}=0$

$\p F_{1}=E_{12}+C_1$ by $R_2$ and $R_3+R_4+R_5$

$\p F_{2}=E_{22}+C_3+A$ by $R_2$, $R_3$ and $R_8$

$\p G_{11}=G_{12}$ by $R_6$

$\p G_{12}=0$

$\p G_{21}=E_{21}+G_{11}+G_{22}$  by $R_4+R_7$, $R_9$ and $R_6$

$\p G_{22}=E_{22}+G_{12}$ by $R_4+R_7$ and $R_9$

$\p G_{31}=D_{21}+G_{32}$ by $R_4$ and $R_6$

$\p G_{32}=D_{22}$ by $R_4$

The generators split into two sets: In the first set we have the
generators $X$, $C_1$, $C_2$, $E_{11}$, $E_{12}$, $F_1$ with the
following boundary maps: $\p X=0$, $\p C_{1}=0$, $\p C_{2}=X+C_1$,
$\p E_{12}=0$, $\p E_{11}=X+E_{12}$, $\p F_{1}=E_{12}+C_1$. Note
that these generators all correspond to the $Spin^c$ structure
$s_\xi$ because we know (\cite{hkm})  that the cycle $X$
corresponds to $s_\xi$, and there are Whitney disks connecting $X$
and $C_2$, $C_2$ and $C_{1}$, $C_1$ and $F_1$, $F_1$ and $E_{12}$,
$E_{12}$ and $E_{11}$. Similarly, there exist Whitney disks connecting the
other 16 generators. Let $V_1$ be the vector space over
$\mathbb{Z}_2$ generated by $X$, $C_1$, $C_2$, $E_{11}$, $E_{12}$,
$F_1$  and let $\p_1 : V_1 \to V_1$ denote the linear map induced
by the boundary maps. Then it is easy to see that rank $\p_1= 2$
and dim $\ker \p_1 =4$. It follows that $\widehat{HF} (-S^1 \times
S^2, s_\xi)=\mathbb{Z}_2 \oplus \mathbb{Z}_2$ which is generated
by $[X]= c(\xi)$ and $[C_2+E_{11}+F_1]$. Hence we conclude that
$c(\xi) \neq 0$.

{\Rem In \cite{p}, Plamenevskaya argues that $\p E_{11}=X+E_{12}$
($d\textbf{x} =  \textbf{c} + \textbf{y}$ in her notation) is
sufficient to show that $[X] \neq 0$. But in fact one has to show
that $E_{12}$ is not a boundary. For a complete argument one has
take into account the boundary relations $\p F_{1}=E_{12}+C_1$ and
$\p C_{2}=X+C_1$.}

To see that the generators in the second set correspond to a different
$Spin^c$ structure $s \neq s_\xi$, consider the loop
$\Gamma$ in the Heegaard surface $\Sigma$ obtained by
concatenating the following paths: part of $\a_1$ from $x_1$ to
$v_2$, part of $\b_2$ from $v_2$ to $x_2$, part of $\a_2$ from
$x_2$ to $z_1$, part of $\b_3$ from $z_1$ to $x_3$, part of $\a_3$
from $x_3$ to $z_2$ and part of $\b_1$ from $z_2$ to $x_1$.
According to \cite{os1}, the difference between the $Spin^c$
structures which correspond to $X=(x_1,x_2,x_3)$ and
$F_2=(v_2,z_1,z_2)$ is measured by the Poincar\'e dual of $p ([\Gamma])$
 in $H^2(S^1 \times
S^2 ; \bfz) $, where $$p : H_1(\Sigma ; \bfz ) \to
\frac{H_1(\Sigma ; \bfz )}{<[\a_1] ,[\a_2] , [\a_3] ,[\b_1],
[\b_2], [\b_3]
>} \cong H_1 (S^1 \times S^2 ; \bfz) \cong \bfz$$ is the quotient
homomorphism. In Figure~\ref{gamma2}, the curve $\Gamma$ is drawn on
the Heegaard surface $\Sigma$ together with $\gamma_i$'s such that
$[\gamma_i]$'s complete $[\a_i]$'s to a basis for the first homology
of $\Sigma$.

\begin{figure}[ht]
  \relabelbox \small {
  \centerline{\epsfbox{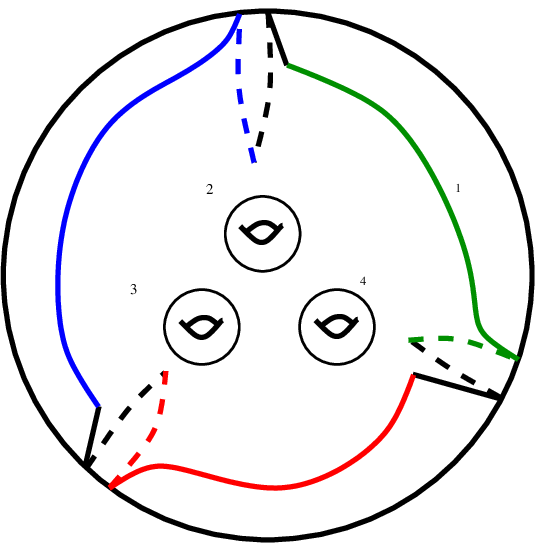}}}
  \relabel{1}{{$\Gamma$}}
  \relabel{3}{{$\gamma_3$}}
  \relabel{2}{{$\gamma_1$}}
 \relabel{4}{{$\gamma_2$}}
\endrelabelbox
        \caption{The curve $\Gamma$ on the Heegaard surface $\Sigma$}
        \label{gamma2}
\end{figure}

\begin{figure}[ht]
  \relabelbox \small {
  \centerline{\epsfbox{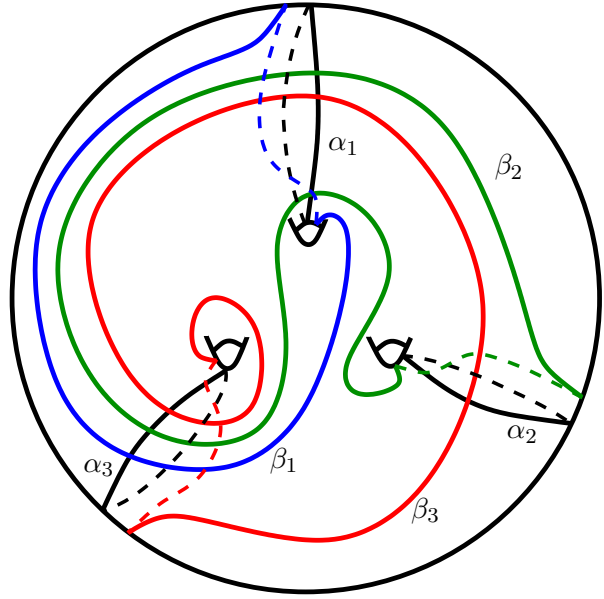}}}
  \relabel{2}{{$\b_1$}}
  \relabel{4}{{$\b_2$}}
  \relabel{6}{{$\b_3$}}
 \relabel{1}{{$\a_1$}}
  \relabel{3}{{$\a_2$}}
  \relabel{5}{{$\a_3$}}

  \endrelabelbox
        \caption{The $\a$ and $\b$ curves on $\Sigma = S_{1/2} \cup
        -S_0$ pictured from the $S_{0}$-side.}
        \label{alphabeta2}
\end{figure}

On one hand $[\Gamma] = [\gamma_1] + [\gamma_2] + [\gamma_3] \in
H_1 (\Sigma ; \bfz)$, where each of these curves is oriented ``clockwise".
On the other hand, the kernel of the
quotient epimorphism $p$ is generated by
$[\gamma_1]+[\gamma_3]$,
$[\gamma_2]+[\gamma_3]$ and $[\alpha_i]$'s (see Figure~\ref{alphabeta2}).
 Therefore $p([\Gamma])$ is $\pm 1 \in \bfz \cong H_1 (S^1 \times S^2 ; \bfz)$,
in particular nonzero. This implies that the generators $X$ and
$F_2$ of the Heegaard Floer chain complex correspond to different
$Spin^c$ structures, i.e. $s \neq s_{\xi}$.

 Let $V_2$ be the vector space generated by
the remaining generators and let $\p_2 : V_2 \to V_2$ denote the
boundary map. One can calculate by simple linear algebra that rank
$\p_2= 8$ and dim $\ker \p_2 =8$. Hence we conclude that the
homology for $(V_2, \p_2)$ is trivial, i.e., $\widehat{HF} (-S^1
\times S^2, s)=0$. Since there are no other generators, the
Heegaard Floer homology groups in the other $Spin^c$ structures
are automatically zero.  Consequently we get
$$\widehat{HF} (-S^1 \times S^2)= \widehat{HF} (-S^1 \times S^2,
s_\xi) \oplus \widehat{HF} (-S^1 \times S^2, s) = \mathbb{Z}_2
\oplus \mathbb{Z}_2,$$ which was indeed proved in \cite{os1}.


\subsection{Basis II}

In the following we choose a different basis on the page $S$
of $\OB_\xi$ and repeat the calculations above. The point is that
with this new choice of basis we will have fewer
generators and fewer relations. We depict the $\alpha$
and $\beta$ curves on page $S_{1/2}$ in Figure~\ref{opbook}.  Now
by examining the intersections of $\alpha$ and $\beta$ curves on
$\Sigma$ we see that there are exactly eight generators of the
Heegaard Floer chain complex: $X=(x_1,x_2,x_3), A=(w_2,y_1,x_3),
B=(x_1,y_2,x_3), C=(x_1,z_1,y_3), D=(w_2,y_3,z_2),
E=(w_1,x_2,z_2), F=(w_1,y_2,z_2), G=(w_1,y_1,z_1) $.

\begin{figure}[ht]
  \relabelbox \small {
  \centerline{\epsfbox{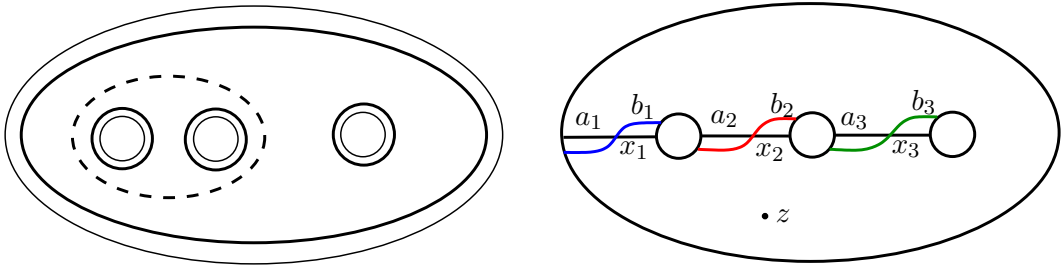}}}
  \relabel{1}{{$a_1$}}
  \relabel{2}{{$a_2$}}
  \relabel{3}{{$a_3$}}
\relabel{4}{{$b_1$}}
  \relabel{5}{{$b_2$}}
  \relabel{6}{{$b_3$}}
\relabel{a}{{$x_1$}}
  \relabel{c}{{$x_3$}}
  \relabel{b}{{$x_2$}}
  \relabel{d}{{$z$}}

  \endrelabelbox
        \caption{Left: Dehn twists about the solid curves are right-handed, while the Dehn twist
        about the dashed curve is left-handed. Right: A basis $\{\a_1,\a_2,\a_3\}$
        on the page $S_{1/2}$, the arcs $\{b_1,b_2, b_3\}$, the intersection points
        $\{x_1, x_2, x_3 \}$, and the base point $z$. }
        \label{opbook}

\end{figure}

\begin{figure}[ht]
  \relabelbox \small {
  \centerline{\epsfbox{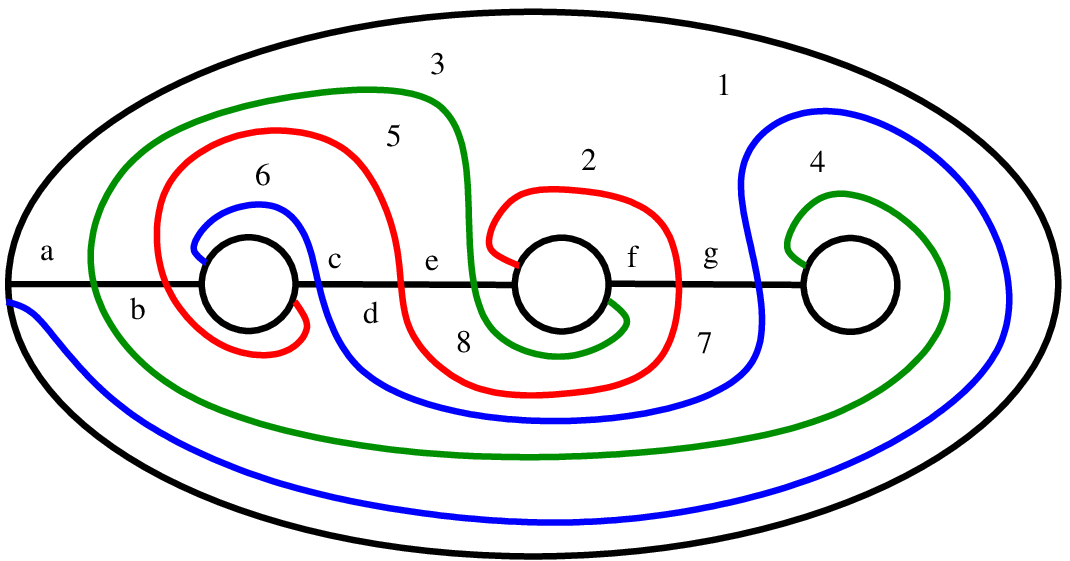}}}
  \relabel{1}{{$\b_1$}}
  \relabel{2}{{$\b_2$}}
  \relabel{3}{{$\b_3$}}
 \relabel{4}{{$1$}}
  \relabel{5}{{$2$}}
  \relabel{6}{{$3$}}
\relabel{7}{{$4$}}
  \relabel{8}{{$5$}}
  \relabel{a}{{$w_1$}}
 \relabel{b}{{$w_2$}}
 \relabel{c}{{$y_1$}}
  \relabel{d}{{$y_2$}}
  \relabel{e}{{$y_3$}}
\relabel{f}{{$z_1$}}
 \relabel{g}{{$z_2$}}

  \endrelabelbox
        \caption{The $\a$ and $\b$ curves on page $S_0$}
        \label{mono}
\end{figure}

There are five regions which do not contain $z$. These are denoted by
$R_1, \ldots, R_5$ and labelled by their indices in Figure~\ref{mono}. Note that
all of the regions are already squares. So we do not need to apply
any finger moves. Below we list the boundary maps induced by these
squares:

$\p X=0$

$\p A=B$ by $R_3$

$\p B=0$

$\p C=B$ by $R_5$

$\p D=A+C$ by $R_3+R_4$ and $R_4+R_5$

$\p E=X+X=0$ by $R_1$ and $R_2+R_3+R_4+R_5$

$\p F=B+D+G+B= D+G$ by $R_1$, $R_2$, $R_4$ and $R_2+R_3+R_4+R_5$

$\p G=A+C$ by $R_2+R_3$ and $R_2+R_5$

The chain complex naturally splits with respect to the $Spin^c$
structures. The generators $X$ and $E$ correspond to the $Spin^c$
structure $s_\xi$ which is uniquely determined by $c_1(s_\xi)=
c_1(\xi)=0$. The other generators correspond to a different
$Spin^c$ structure $s \neq s_\xi$. To see this consider the loop
$\Gamma$ in the Heegaard surface $\Sigma$ obtained by
concatenating the following paths: part of $\a_1$ from $x_1$ to
$w_1$, part of $\b_3$ from $w_1$ to $y_3$, part of $\a_2$ from
$y_3$ to $x_2$, part of $\b_2$ from $x_2$ to $z_1$, part of $\a_3$
from $z_1$ to $z_2$ and part of $\b_1$ from $z_2$ to $x_1$.
According to \cite{os1}, the difference between the $Spin^c$
structures which correspond to $C=(x_1,y_3,z_1)$ and
$E=(w_1,x_2,z_2)$ is measured by the Poincar\'e dual of $p ([\Gamma])$
in $H^2(S^1 \times S^2 ; \bfz) $, where $$p : H_1(\Sigma ; \bfz ) \to
\frac{H_1(\Sigma ; \bfz )}{<[\a_1] ,[\a_2] , [\a_3] ,[\b_1],
[\b_2], [\b_3]
>} \cong H_1 (S^1 \times S^2 ; \bfz) \cong \bfz$$ is the quotient
homomorphism. In Figure~\ref{gamma}, the curve $\Gamma$ is drawn on
the Heegaard surface $\Sigma$ together with $\gamma_i$'s such that
$[\gamma_i]$'s complete $[\a_i]$'s to a basis for the first homology
of $\Sigma$.
\begin{figure}[ht]
  \relabelbox \small {
  \centerline{\epsfbox{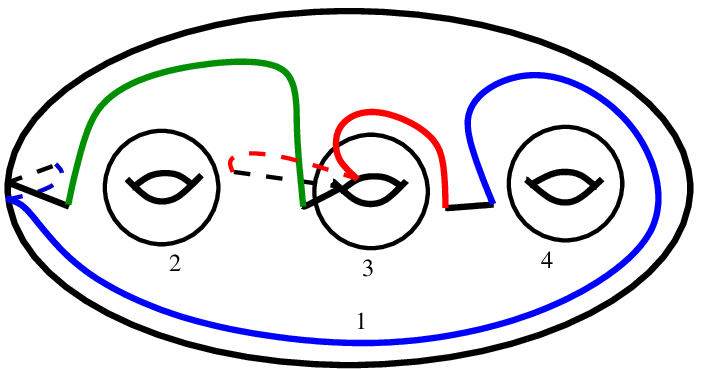}}}
  \relabel{1}{{$\Gamma$}}
  \relabel{2}{{$\gamma_1$}}
  \relabel{3}{{$\gamma_2$}}
 \relabel{4}{{$\gamma_3$}}

  \endrelabelbox
        \caption{The curve $\Gamma$ on the Heegaard surface $\Sigma$}
        \label{gamma}
\end{figure}

\begin{figure}[ht]
  \relabelbox \small {
  \centerline{\epsfbox{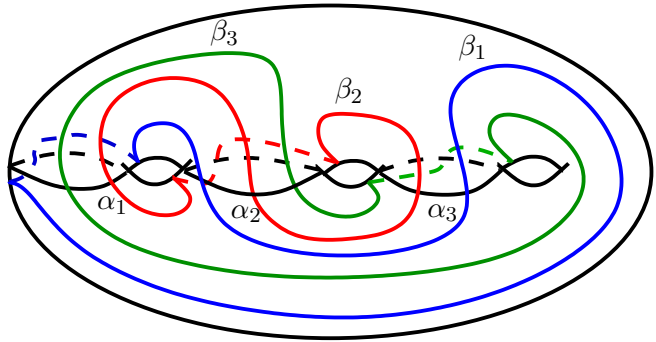}}}
  \relabel{2}{{$\b_1$}}
  \relabel{4}{{$\b_2$}}
  \relabel{6}{{$\b_3$}}
 \relabel{1}{{$\a_1$}}
  \relabel{3}{{$\a_2$}}
  \relabel{5}{{$\a_3$}}

  \endrelabelbox
        \caption{The $\a$ and $\b$ curves on $\Sigma = S_{1/2} \cup
        -S_0$ pictured from the $S_{0}$-side.}
        \label{alphabeta}
\end{figure}

On one hand $[\Gamma] = [\gamma_1] + [\gamma_2] + [\gamma_3] \in
H_1 (\Sigma ; \bfz)$, where each of these curves is oriented ``clockwise".
On the other hand, the kernel of the
quotient epimorphism $p$ is generated by $[\gamma_1]+[\gamma_3]$,
$[\gamma_1]-[\gamma_2]$ and $[\alpha_i]$'s (see Figure~\ref{alphabeta}).
Therefore $p([\Gamma])$ is $\pm 1 \in \bfz \cong H_1 (S^1 \times S^2 ; \bfz)$,
in particular nonzero. This implies that the generators $C$ and
$E$ of the Heegaard Floer chain complex correspond to different
$Spin^c$ structures, i.e. $s \neq s_{\xi}$. Moreover one can see
that the homology induced by the generators $\{A,B,C,D,F,G\}$ is
trivial. Therefore we conclude that $\widehat{HF} (-S^1 \times
S^2)= \widehat{HF} (-S^1 \times S^2, s_\xi)=\mathbb{Z}_2 \oplus
\mathbb{Z}_2$ which is generated by $[X]$ and $[E]$. This confirms
again that the contact class $[X]=c(\xi) \neq 0$.


\subsection{Basis III}

Interestingly there is yet another basis which simplifies the
calculations dramatically. The basis given in
Figure~\ref{opbook2} produces only two generators $X=(x_1,x_2,x_3)$,
and $A=(y_1,x_2,y_3)$. Other than the one which contains the base
point $z$, there are four regions $R_1, \dots , R_4$ indicated in
Figure~\ref{mono2} by their indices and these regions are already
squares. Moreover $\p A = X+X = 0$ by $R_1+R_2$ and $R_3+R_4$, and
$\p X=0$, confirming $\widehat{HF} (-S^1 \times S^2)= \widehat{HF}
(-S^1 \times S^2, s_\xi)=\mathbb{Z}_2 \oplus \mathbb{Z}_2$, the
nontriviality of the contact class $[X]=c(\xi)$ and consequently
the tightness of $\xi$.

\begin{figure}[ht]
  \relabelbox \small {
  \centerline{\epsfbox{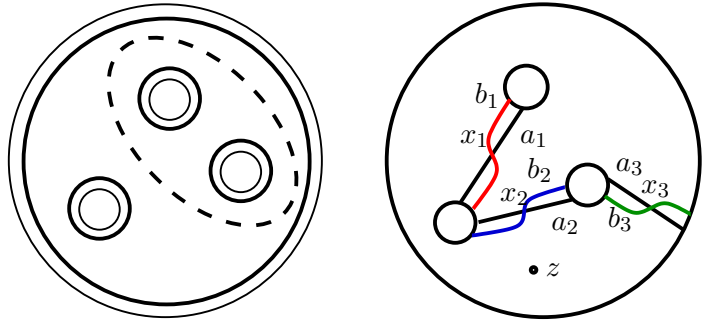}}}
  \relabel{2}{{$a_1$}}
  \relabel{4}{{$a_2$}}
  \relabel{6}{{$a_3$}}
\relabel{1}{{$b_1$}}
  \relabel{3}{{$b_2$}}
  \relabel{5}{{$b_3$}}
  \relabel{7}{{$x_1$}}
  \relabel{8}{{$x_2$}}
  \relabel{9}{{$x_3$}}
\relabel{a}{{$z$}}
  \endrelabelbox
        \caption{Left: Dehn twists about the solid curves are right-handed, while the Dehn twist
        about the dashed curve is left-handed. Right: A basis $\{a_1,a_2,a_3\}$
        on the page $S_{1/2}$, the arcs $\{b_1,b_2, b_3\}$, the intersection points
        $\{x_1, x_2, x_3 \}$, and the base point $z$.}
        \label{opbook2}

\end{figure}

\begin{figure}[ht]
  \relabelbox \small {
  \centerline{\epsfbox{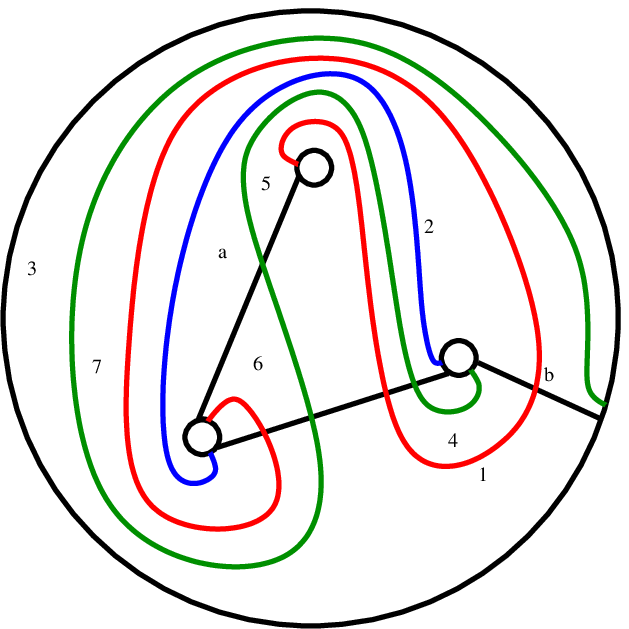}}}
  \relabel{1}{{$\b_1$}}
  \relabel{2}{{$\b_2$}}
  \relabel{3}{{$\b_3$}}
 \relabel{4}{{$1$}}
  \relabel{5}{{$2$}}
  \relabel{6}{{$3$}}
\relabel{7}{{$4$}}
  \relabel{a}{{$y_1$}}
 \relabel{b}{{$y_3$}}

  \endrelabelbox
        \caption{}
        \label{mono2}
\end{figure}


\section{An overtwisted contact structure on $S^3$}

Consider the contact $3$-manifold $(Y, \xi)$ described by the
surgery diagram depicted in Figure~\ref{s3}. When we convert this
diagram into a smooth diagram (see Figure~\ref{sms3}) and blow
down the $-1$-curve, we immediately see that the underlying
$3$-manifold $Y$ is homeomorphic to $S^3$. Note that there is a
unique $Spin^c$ structure on $S^3$. From the contact surgery
diagram we obtain an open book decomposition $OB_{\xi}$ on $S^3$
compatible with $\xi$ shown on the right of Figure~\ref{opbook3}.

\begin{figure}[ht]
  \relabelbox \small {
  \centerline{\epsfbox{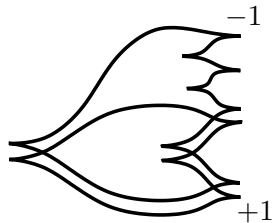}}}
  \relabel{1}{{$-1$}}
  \relabel{2}{{$+1$}}

  \endrelabelbox
        \caption{A contact surgery diagram}
        \label{s3}
\end{figure}

\begin{figure}[ht]
  \relabelbox \small {
  \centerline{\epsfbox{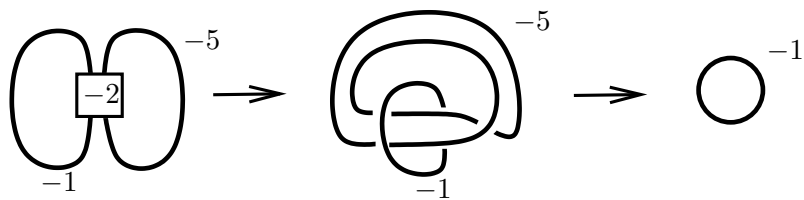}}}
  \relabel{1}{{$-1$}}
  \relabel{2}{{$-5$}}
 \relabel{3}{{$-1$}}
 \relabel{4}{{$-2$}}
 \relabel{5}{{$-1$}}
 \relabel{6}{{$-5$}}

  \endrelabelbox
        \caption{Underlying $3$-manifold is $S^3$}
        \label{sms3}
\end{figure}

\begin{figure}[ht]
  \relabelbox \small {
  \centerline{\epsfbox{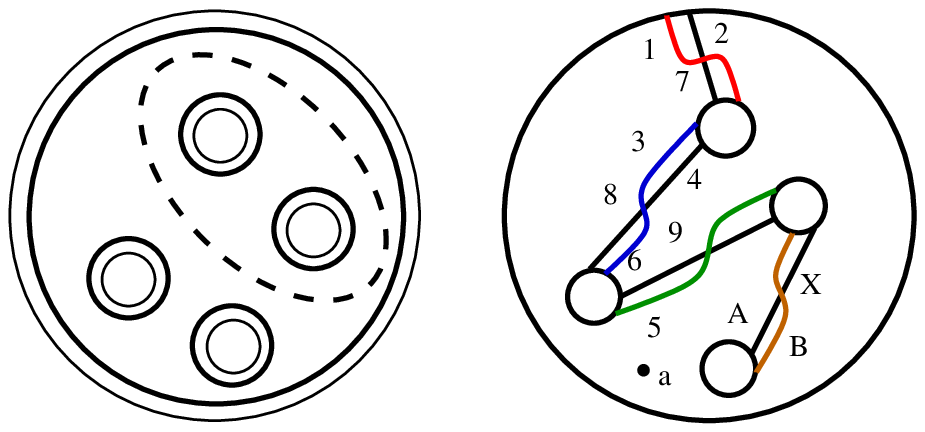}}}
  \relabel{2}{{$a_1$}}
  \relabel{4}{{$a_2$}}
  \relabel{6}{{$a_3$}}
\relabel{1}{{$b_1$}}
  \relabel{3}{{$b_2$}}
  \relabel{5}{{$b_3$}}
  \relabel{7}{{$x_1$}}
  \relabel{8}{{$x_2$}}
  \relabel{9}{{$x_3$}}
\relabel{a}{{$z$}}
\relabel{A}{{$a_4$}}
\relabel{B}{{$b_4$}}
\relabel{X}{{$x_4$}}
  \endrelabelbox
        \caption{Left: Dehn twists about the solid curves are right-handed, while the Dehn twist
        about the dashed curve is left-handed. Right: A basis $\{a_1,a_2,a_3, a_4\}$
        on the page $S_{1/2}$, the arcs $\{b_1,b_2, b_3, b_4\}$, the intersection points
        $\{x_1, x_2, x_3, x_4 \}$, and the base point $z$.}
        \label{opbook3}

\end{figure}

\begin{figure}[ht]
  \relabelbox \small {
  \centerline{\epsfbox{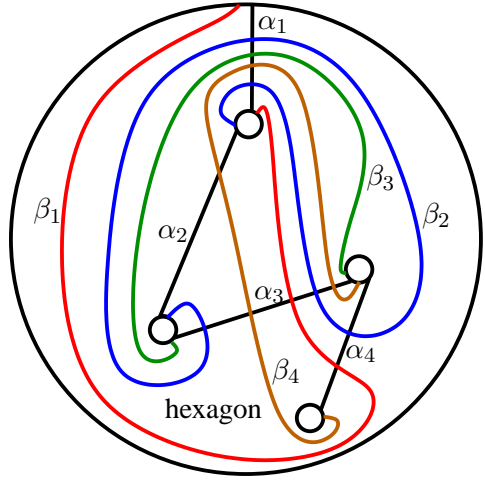}}}
  \relabel{N}{{hexagon}}
\relabel{3}{{$\b_3$}}
  \relabel{2}{{$\b_2$}}
  \relabel{1}{{$\b_1$}}
\relabel{5}{{$\a_1$}}
 \relabel{6}{{$\a_2$}}
 \relabel{7}{{$\a_3$}}
   \relabel{4}{{$\b_4$}}
\relabel{8}{{$\a_4$}}
  \endrelabelbox
        \caption{The bad region is indicated on page $S_0$}
        \label{nonsquare2}

\end{figure}

\begin{figure}[ht]
  \relabelbox \small {
  \centerline{\epsfbox{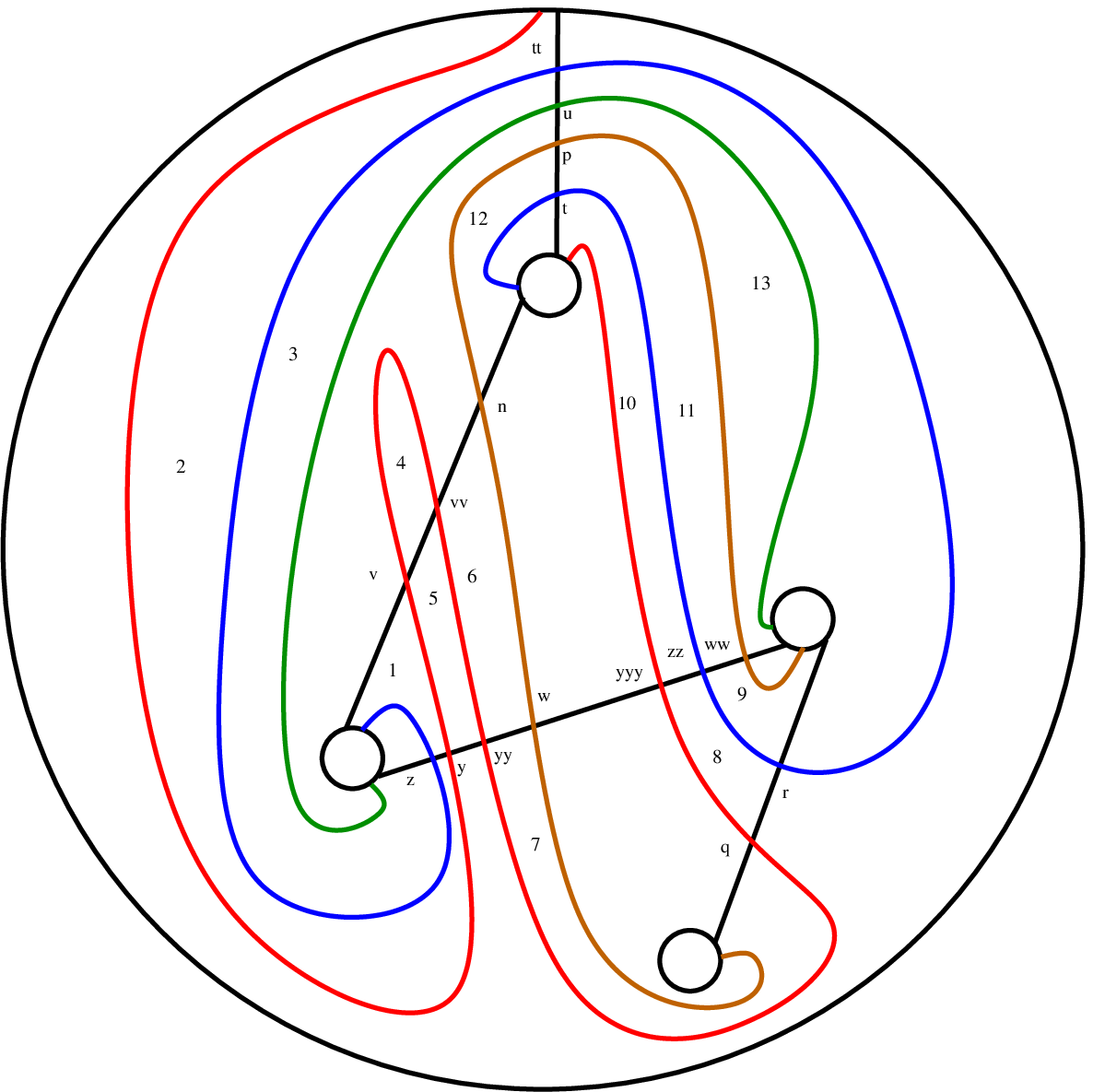}}}
  \relabel{1}{{$1$}}
  \relabel{2}{{$2$}}
  \relabel{3}{{$3$}}
 \relabel{4}{{$4$}}
  \relabel{5}{{$5$}}
  \relabel{6}{{$6$}}
\relabel{7}{{$7$}}
  \relabel{8}{{$8$}}
 \relabel{9}{{$9$}}
 \relabel{10}{{$10$}}
  \relabel{11}{{$11$}}
\relabel{12}{{$12$}}
  \relabel{13}{{$13$}}

 \relabel{y}{{$y_1$}}
  \relabel{yy}{{$y_2$}}
  \relabel{yyy}{{$y_3$}}
 \relabel{z}{{$z_1$}}
  \relabel{zz}{{$z_2$}}
  \relabel{w}{{$w_1$}}
\relabel{ww}{{$w_2$}}
  \relabel{t}{{$t_1$}}
 \relabel{tt}{{$t_2$}}
 \relabel{u}{{$u$}}
  \relabel{p}{{$p$}}
\relabel{q}{{$q$}}
  \relabel{r}{{$r$}}
   \relabel{n}{{$n$}}
  \relabel{v}{{$v_1$}}
\relabel{vv}{{$v_2$}}

  \endrelabelbox
        \caption{Finger move}
        \label{mono3}
\end{figure}

Choosing a basis indicated on the right in Figure~\ref{opbook3}
gives the Heegaard diagram whose $\a$ and $\b$ curves are shown in
Figure~\ref{nonsquare2}. It is possible to convert this Heegaard
diagram into one without any bad region (except for the
region including the base point $z$), by a simple finger move. The
curves of this new Heegaard diagram are depicted in
Figure~\ref{mono3}. There are 13 regions which do not contain $z$.
These regions are denoted by $R_1, \dots , R_{13}$ and
labelled by their indices in
Figure~\ref{mono3}. Examining the intersections of $\a$ and
$\b$ curves on the Heegaard surface $\Sigma$ one can confirm that
the Heegaard Floer chain complex $\widehat{CF}(\Sigma , \b , \a ,
z)$ have 29 generators in total: $X=(x_1,x_2,x_3), A=(x_1,r,x_3,n)$,
$B_{ij}=(v_i,t_j,x_3,x_4), C_{ij}=(v_i,z_j,u,x_4), D_{ij} =
(v_i,r,u,w_j), E_i = (v_i , r, x_3, p)$, $F_k =(y_k, x_2, u, x_4)$,
$G_k= (y_k, r, u, n), H_i = (q,t_i, x_3, n)$, $I_i = (q,x_2,u,w_i),
J=(q,x_2,x_3,p)$, $K_i = (q,z_i, u,n)$, where $1 \leq i,j, \leq 2$
and $1\leq k \leq 3$. One can also calculate all the boundary maps:

$\p X = 0$

$\p A = X$ by $R_{9}+R_{11} + R_{12} $

$\p B_{11} = B_{21}$ by $ R_{4}$

$\p B_{12} = X + B_{22}$ by $ R_{1} + R_{2} $ and $R_{4}$

$\p B_{2j} = 0$

$\p C_{11} = B_{12} + C_{21} + F_1 $ by $R_{3}$, $R_{4}$ and $R_{1}$

$\p C_{12} = B_{11} + C_{22} $ by $ R_{11} + R_{13} $ and $R_{4}$

$\p C_{21} = B_{22} + F_2 $ by $ R_{3} $ and $R_{1}+R_5$

$\p C_{22} = B_{21} $ by $ R_{11}+R_{13}$

$\p D_{11} = D_{21} $ by $ R_{4}$

$\p D_{12} = C_{12} + D_{22} + E_1 $ by $R_{9}$, $R_{4}$ and $R_{13}$

$\p D_{21} = 0 $

$\p D_{22} = C_{22} +E_2 $ by $ R_{9} $ and $R_{13}$

$\p E_1 = B_{11} + E_2 $ by $ R_{9} + R_{11} $ and $R_{4}$

$\p E_2 = B_{21} $ by $ R_9 + R_{11}$

$\p F_1 = X+F_2 $ by $ R_{2} + R_{3} $ and $R_{4} + R_5 $

$\p F_2 = 0$

$\p F_3 = X$ by $ R_{10}+R_{11} + R_{13}$

$\p G_1 = A+D_{11}+F_1 +G_2$ by $R_2+ R_{3}$, $R_{5}+R_6$, $R_9+R_{11}+R_{12}$ and $R_{4}+R_5$

$\p G_2 = D_{21} +F_2 $ by $ R_{6} $ and $R_9+ R_{11}+R_{12}$

$\p G_3 = A + F_3 $ by $R_{10}+ R_{11}+R_{13}$ and $R_9+ R_{11}+R_{12}$

$\p H_1 = A + B_{21} + J $ by $R_8+ R_{10}$, $R_{6}+R_7$ and $R_{12}$

$\p H_2 = B_{22} $ by  $R_{6}+R_7$

$\p I_1 = F_2 $ by $ R_{7}$

$\p I_2 = F_3 + J$ by $ R_{8} + R_{9} $ and $R_{13}$

$\p J = X$ by $R_8+ R_9+ R_{10}+R_{11}$

$\p K_1 = C_{21} + H_2 + I_1 $ by $R_{6}+R_7$, $R_{3}$ and $R_{1}+R_5+R_6$

$\p K_2 = C_{22} + G_3 + H_1 + I_2 $ by $R_{6}+R_7$, $R_{8}$, $R_{11}+R_{13}$ and $R_{11}+R_{12}$

From $\p A = X$ one immediately sees that $c(\xi)= [X] = 0 \in
\widehat{HF}(-S^3, s_{\xi} )$ even with $\bfz$ coefficients.
By an important result in \cite{os}, $\xi$ is not
Stein fillable. In fact, since the unique tight contact structure
on $S^3$ is Stein fillable by \cite{elias}, $\xi$ is overtwisted. On the other
hand, it is seen that the image of the boundary map is 14
dimensional since it is generated by $\{ X,B_{21} , B_{22},
D_{21}, F_2 , A+F_3, A+J, B_{11}+C_{22}, B_{11} + E_2,
B_{12}+C_{21} + F_1, A+D_{11}+F_1+G_2, C_{12}+D_{22}+E_1, C_{21}+
H_2 + I_1, C_{22} + G_3+H_1+I_2 \}$. Therefore the kernel is
$29-14=15$ dimensional. Hence we verified that $\widehat{HF}
(-S^3, s_{\xi} ) = \widehat{HF} (-S^3) = \bfz_2$.

{\Rem Note that this contact structure has an open book
decomposition which differs from the one in the previous section
by an additional puncture and a right-handed Dehn twist around
that puncture. It is interesting that these modifications, even
though the Dehn twist is right-handed, convert a Stein fillable
contact structure to an overtwisted one.}

{\Rem An alternative way to see the overtwistedness of the contact structure $\xi$ given
by the contact surgery diagram in Figure~\ref{s3} is to use the $d_3$ invariant of
$\xi$ as a plane field and compare it with that of the unique tight
contact structure on $S^3$. The former is $1/2$ whereas the latter is $-1/2$. }

\vspace{0.2in}

\noindent{\bf {Acknowledgement}}: We would like to thank Andr\'as Stipsicz
for comments on a draft of this paper. TE was partially supported by a
GEBIP grant of the Turkish Academy of Sciences and a CAREER grant
of the Scientific and Technological Research Council of Turkey. BO
was partially supported by a research grant of the Scientific and
Technological Research Council of Turkey.



\begin{thebibliography}{99999}

\bibitem{ao}
S. Akbulut and B. Ozbagci, {\em Lefschetz fibrations on compact
Stein surfaces}, Geom. Topol. {\bf 5} (2001), 319--334.

\bibitem{a} M. F. Arikan,
{\em On the support genus of a contact structure given by a
surgery diagram,} preprint, arXiv:math.GT/0704.1670

\bibitem{dg}
F. Ding and H. Geiges, {\em A Legendrian surgery presentation of
contact 3-manifolds}, Math. Proc. Cambridge Philos. Soc. {\bf136}
(2004), 583--598.

\bibitem{dgs}
F. Ding, H. Geiges and  A. Stipsicz, {\em Lutz twist and contact
surgery,} Asian J. Math. 9 (2005), no. 1, 57--64.

\bibitem{elias}
Y. Eliashberg, {\em Filling by holomorphic discs and its applications,} Geometry of low-dimensional manifolds, 2 (Durham, 1989), 45--67, London Math. Soc. Lecture Note Ser., 151, Cambridge Univ. Press, Cambridge, 1990.

\bibitem{e}
J. B. Etnyre, {\em Lectures on open book decompositions and
contact structures,} Floer homology, gauge theory, and
low-dimensional topology, 103--141, Clay Math. Proc., 5, Amer.
Math. Soc., Providence, RI, 2006.

\bibitem{eo}
J. B. Etnyre and B. Ozbagci, {\em Open books and plumbings,} Int.
Math. Res. Notices vol. 2006, Article ID 72710, 2006.

\bibitem{g}
E. Giroux, \emph{Contact geometry: from dimension three to higher dimensions,}
Proceedings of the International Congress of Mathematicians (Beijing 2002), 405--414.

\bibitem{hkm}
K. Honda, W. Kazez and  G. Mati\'{c}, {\em On the contact class in
Heegaard Floer homology,} preprint, arXiv:math.GT/0609734

\bibitem{ls}
P. Lisca and  A. I. Stipsicz, {\em Seifert fibered contact
three-manifolds via surgery,} Algebr. Geom. Topol. 4 (2004),
199--217 (electronic).

\bibitem{ozst}
B. Ozbagci and A. I. Stipsicz, {\em Surgery on contact
3--manifolds and Stein surfaces}, Bolyai Soc. Math. Stud., Vol.
{\bf 13}, Springer, 2004.

\bibitem{os1} P. Ozsv\'{a}th and Z. Szab\'{o},
{\em Holomorphic disks and topological invariants for closed
three-manifolds,} Ann. of Math. (2) 159 (2004), no. 3, 1027--1158.

\bibitem{os} P. Ozsv\'{a}th and Z. Szab\'{o},
{\em Heegaard Floer homology and contact structures,} Duke Math.
J. 129 (2005), no. 1, 39--61.



\bibitem{p1}
O. Plamenevskaya, {\em Contact structures with distinct Heegaard
Floer invariants}, Math. Res. Lett. 11 (2004), no. 4, 547--561.

\bibitem{p}
O. Plamenevskaya, {\em A combinatorial description of the Heegaard
Floer contact invariant}, to appear in Algebr. Geom. Topol.,
arXiv:math.GT/0612322


\bibitem{sw}
S. Sarkar and J. Wang, {\em An algorithm for computing some
Heegaard Floer homologies}, preprint, arXiv:math.GT/0607777

\bibitem{s}
S. Sch\"onenberger, {\em Planar open books and symplectic
fillings,} PhD thesis, University of Pennsylvania, 2005.

\bibitem{st}
A. I. Stipsicz, {\em Surgery diagrams and open book decompositions
of contact 3-manifolds,} Acta Math. Hungar. 108 (2005), no. 1-2,
71--86.

\end{thebibliography}
\end{document}